\documentclass[11pt]{article}
\usepackage{cite} 
\usepackage{amsmath,amsthm,amsfonts,amssymb,mathrsfs,epsfig,bm}
\usepackage{titlesec,fancyhdr,titletoc}

\usepackage[usenames]{color}
\usepackage[colorlinks=true]{hyperref}

\parskip        \smallskipamount
\newtheorem{theorem}{Theorem}
\newtheorem{lemma}{Lemma}
\newtheorem{corollary}{Corollary}
\newtheorem{proposition}{Proposition}
\newtheorem{definition}{Definition}
\newtheorem*{theorem*}{Theorem}
\theoremstyle{remark}

\newtheorem{remark}{Remark}

\def\N{\mathbb{N}}
\def\Z{\mathbb{Z}}

\def\R{\mathbb{R}}

\def\P{\mathbb{P}}
\def\E{\mathbb{E}}
\def\FF{\mathscr{F}}
\def\GG{\mathscr{G}}

\def\SS{\mathcal{S}}

\def\NN{\mathcal N}

\renewcommand{\phi}{\varphi}
\renewcommand{\epsilon}{\varepsilon}
\newcommand{\1}{{\text{\Large $\mathfrak 1$}}}

\newcommand{\var}{\operatorname{var}}

\renewcommand{\liminf}{\varliminf}

\newcommand{\eqdist}{\stackrel{\text{\rm (d)}}{=}}

\definecolor{mygray}{gray}{0.9}
\definecolor{deeppink}{RGB}{255,20,147}
\definecolor{mygreen}{rgb}{0.05, 0.576, 0.03}
\definecolor{myred}{rgb}{0.768, 0.09, 0.09}

\long\def\symbolfootnote[#1]#2{\begingroup
\def\thefootnote{\fnsymbol{footnote}}\footnote[#1]{#2}\endgroup}
\newcommand{\keywords}[1]{ \noindent {\footnotesize
             {\small \em Keywords and phrases.} {\sc #1} } }

\newcounter{para}
\newcommand{\para}{ \refstepcounter{para} \paragraph{\thepara.} }
\newcommand{\resetpara}{\setcounter{para}{0}}

\newcommand{\paratitle}[1]{ \refstepcounter{para} \paragraph{\thepara. #1} }

\def\WW{\widetilde W}
\usepackage{url}
\def\RR{\mathcal{R}}

\def\convd{ \xrightarrow[]{\text{(d)}} }
\def\sog{\operatorname{SOG}}

\def\ccm{\operatorname{CCM}}
\def\pwit{\operatorname{PWIT}}

\def\UU{\mathcal{U}}

\usepackage{cancel}

\begin{document}

\title{\Large \bf Limiting properties of random graph models with vertex
and edge weights} 
\author{\sc Sergey Foss
\thanks{Heriot-Watt University and Novosibirsk State University;
S.Foss@hw.ac.uk; Research supported by RSF grant 17-11-01173}
 \and \sc Takis Konstantopoulos
\thanks{Department of Mathematical Sciences, The University of Liverpool, Peach Street,
Liverpool L69 7ZL, UK; takiskonst@gmail.com; Research supported by
Swedish Research Council grant 2013-4688}
}
\date{\small \em 1 August 2018}
\maketitle

\begin{abstract}
This paper provides an overview of results, concerning 
longest or heaviest paths, in the area of random directed graphs
on the integers  along with some extensions.
We study first-order asymptotics of heaviest paths allowing weights both on edges
and vertices and assuming that weights on edges are signed.  
We aim at an exposition that summarizes, simplifies, and extends proof ideas.
We also study sparse graph asymptotics, showing convergence
of the weighted random graphs to a certain weighted graph that
can be constructed in terms of Poisson processes. We are motivated
by numerous applications, ranging from ecology to parallel computing model.
It is the latter set of applications that necessitates the introduction of vertex weights.
Finally, we discuss some open problems and research directions.
\\[1mm]
\end{abstract}

\symbolfootnote[0]{To appear in {J.\ Stat.\ Phys.}\ {\tt
https://doi.org/10.1007/s10955-018-2080-3}; the current version may differ from
the published one in some minor points.}

\section{Introduction and background}


The well-known Erd\H{o}s-R\'enyi random graph model \cite{BOLbook}
has an ordered version introduced in \cite{BE84} by Barak and Erd\H{o}s.
Declare a pair $(i,j)$ of integers, $1\le i < j \le n$, an edge
with probability $p$, independently from pair to pair.
The random directed graph thus constructed, being so natural,
has emerged in many areas of applied science.
In {\em mathematical biology}
 (ecology), such graphs are used to model {\em community food
webs} \cite{CBN90,COHNEW91}. The interest in this area is the longest path
in the graph as this is an abstraction of the longest food chain upon
which survivability of a biological population depends. 
Asymptotics for this were obtained by Newman \cite{NEWM92} in the regime
when $p\to 0$ at a certain rate. 

Independently, the graph emerged
as a model in the area of performance evaluation of {\em parallel
computing systems} \cite{GNPT86,ISONEW94} where vertices represent various
tasks whereas edges represent precedence constraints. The longest
path then represents the total execution time  when tasks have identical
processing times.
These papers actually also discuss the more realistic case when
random weights, representing task execution times, are given to the vertices. 
Weights on the edges play a secondary role in this set of applications and this
is what motivated us to make this extension that actually turns out to be
easy to handle.

An infinite version of the model is the graph analyzed in \cite{FK03}:
Let $\Z$ be the set of vertices and let $(i,j)$, $i<j$, be an edge with 
probability $p$, independently from edge to edge. We used the 
term {\em ``stochastic  ordered graph''} for this model in \cite{FK03}
and we shall use the abbreviation $\sog(\Z,p)$ for it in the current paper.
Our motivation in \cite{FK03} was a {\em queueing system} with precedence
constraints and identical  service times, arriving randomly over time.
The stochastic stability of this system (i.e., convergence in
distribution of the
state, such as number of customers in the queue  at time $t$ as $t \to \infty$)
depends on the asymptotic growth of the length $L_n$ of
the longest path in this graph between two vertices at distance $n$.

An application in {\em algebra} was considered by Alon et al.\ \cite{ABBJ94}
where the graph is seen as defining a partial order
on $\{1,\ldots,n\}$ and the question of interest is the number of
linear extensions of the random partial order.

Yet another application of a continuous-vertex extension of $\sog(\Z,p)$
appears in the {\em physics} literature: Itoh and Krapivsky \cite{IK12} introduce
a version, called {\em ``continuum cascade model''} 
of the stochastic ordered graph with set of vertices in $\R_+$
and study asymptotics for the length of longest paths between $0$ and $t>0$,
deriving recursive integral equations for its distribution.

In \cite{FK03} we actually studied a more general version of
$\sog(\Z,p)$ and allowed dependence between the Bernoulli random variables
defining connectivities between edges. Putting the graph in a stationary
and ergodic framework, we showed that the longest length $L_n(p)$
satisfies
\[
\frac{L_n(p)}{n} \to C^0(p), \text{ as } n \to \infty \text{ a.s.\ and in $L^1$,}
\]
for some constant $C^0(p)$. Estimating the constant $C^0(p)$ is essential 
in all areas of applications mentioned above. In a general stationary-ergodic
framework, such estimates are not available. However, for the $\sog(\Z,p)$
model, we were able to reduce the question of estimating $C^0(p)$ to a 
question of analyzing the behavior of an interacting particle system
that we referred to as the {\em ``infinite bin model''}. Using 
{\em extended renovation theory} and Markov chain analysis, we were able
to obtain sharp computable bounds for $C^0(p)$ for all $p$.
In particular, we showed that
\[
C^0(1-q) = 1-q+q^2-3q^3+7q^4 + O(q^5) \text{ as } q \to 0,
\]
that is, in the dense graph regime.
More recently, Mallein and Ramassamy \cite{MR} and \cite{MR2},
using coupling between Barak-Erd\H{o}s graphs and infinite bin models,
managed to provide a full analytic expansion for $C^0(p)$ when $p>0$,
thus completing the last display. They also showed that $C^0(p)$
has first but not second derivative at $p=0$.

In \cite{DFK12}, a further extension of $\sog(\Z,p)$ was given,
one where the edge probabilities depend on the physical distance between
the endpoints. Doing so, we managed to simplify the arguments of 
\cite{FK03} that led to a law of large numbers (LLN) and a central limit
theorem (CLT) for $L_n(p)$.
In a more recent paper \cite{FMS}, the $\sog(\Z,p)$ model was extended
by adding i.i.d.\ random weights to the edges.
A rather tedious extension of the arguments of \cite{DFK12}
was devised in \cite{FMS} in order that both LLN and CLT be derived.
The possibility that weights be heavy-tailed, leading to a
different behavior of the graph, was also studied.

Studying longest paths is also the subject in {\em last passage percolation}
problems in probability theory. In this area, one is given a random directed
graph (think of $\Z_+ \times \Z_+$ with edges $(x,y)$ where
$x=(x_1,x_2)$ is below $y=(y_1,y_2)$ component-wise) and random weights on
the vertices. The weight of a path is the sum of the weights of its vertices.
One is interested in studying the weight of heaviest path contained
in a finite chunk of the graph of ``size'' $n$, as $n \to \infty$.
In a seminal paper, Johansson \cite{JOHANSSON00} considered the last passage
percolation problem on $\Z_+ \times \Z_+$ with i.i.d.\ geometrically
distributed weights on the vertices and obtained that its scaled
fluctuations from its mean converge in distribution to the 
Tracy-Widom distribution appearing in random matrix theory.

Motivated by the last paper, we studied, in \cite{DFK12},
the question of longest paths in a  {\em stochastic ordered ``slabgraph''}
of width $N$ and showed that the asymptotic fluctuations converge
in distribution to the distribution of the largest eigenvalue of 
a random $N \times N$ matrix in the GUE ({\em Gaussian Unitary Ensemble})
\cite{BS2005,BAI05},
The question of what happens when $N \to \infty$ was considered in 
\cite{KT13} and, again, convergence to the Tracy-Widom distribution
\cite{TW93,TW94}
was shown.

In this paper we study an extension of 
$\sog(\Z,p)$ with weights \underline{both} on the vertices and the edges.
We allow weights on the edges to take negative values. 
In doing so, we review the techniques established in the literature,
simplify some of the arguments and unify several results.
Note that introducing weights on both vertices and edges introduces dependencies
between paths that share common vertices.
Let $u$ and $v$ denote random variables representing typical 
edge and vertex weights, respectively.
In fact, we let $u_{i,j}$, $v_i$, $i, j \in \Z$, $i < j$, be independent random
variables where the $u_{i,j}$ all have the distribution of $u$ and the $v_i$ the
distribution of $v$. The weighted graph is obtained by letting, as before,
$(i,j)$ be an edge with probability $p$ but, in addition, we assign weight $u_{i,j}$ to
$(i,j)$. We also assign weight $v_i$ to each $i \in \Z$.
We denote by $\sog(\Z,p,u,v)$ the corresponding graph;
see Section \ref{themodel}. The graph described in the previous paragraphs was 
the graph $\sog(\Z,p,1,0)$, that is, each edge, existing with probability $p$,
is counted as having weight $1$, whereas vertices have no weights
\cite{ABBJ94,BE84,CBN90,COHNEW91,DFK12,FK03,MR,MR2}.
The graph $\sog(\Z, p, u,0)$ was first considered in \cite{FMS}.
Note that the special case $\sog(\Z,1,1,0)$ also makes sense and is also
considered in \cite{FMS}
The graph $\sog(\Z,p,1,v)$ is the one that was essentially introduced
in the work of \cite{GNPT86}. It should be noted that the area of performance evaluation
of parallel processing system is vast and it is not our intention to overview the it.

What we do next is this: 
We prove the strong law of large numbers (SLLN) assuming that 
the  vertex weight $v$ is a.s.\ positive with finite expectation,
the edge weight $u$ has positive and finite expectation and that its
positive part has finite variance (the negative part may have infinite variance).
We then prove a functional central limit theorem (CLT) assuming positive
edge and vertex weights with finite second and third moment, respectively.
We then study the sparse graph limit of the whole random graph,
showing that it becomes a random weighted tree, a weighted version
of the so-called {\em Poisson Weighted Infinite Tree} (PWIT) introduced
by Aldous and Steele \cite{AS} for the study of combinatorial
optimization problems. We provide simple arguments of why the limit
should be so and also discuss equations satisfied by functional of
the limiting random tree.
Finally, we devote the last section to discussing a number
of open and exciting new problems that we believe are of interest in 
several areas of applications of engineering, biology, computer science,
stochastic networks, and statistical physics.


\section{The model with weights on both edges and vertices}
\label{themodel}
We use the notation $\sog(\Z,p)$ for the directed random graph
(stochastic ordered graph) on the set of integers 
$\Z$ obtained
by letting $(i,j)$, $i< j$, be a directed edge with probability $p$.
(We may replace $\Z$ by any totally ordered countable set and we shall
later have occasion to do so.)
This is done independently from edge to edge.
We point out that the restriction of 
$\sog(\Z,p)$ on a finite interval
 is  an ordered-version of the Erd\H{o}s-R\'enyi graph \cite{BE84}.

It is convenient to denote the presence of an edge $(i,j)$, $i < j$, by
\[
\alpha_{i,j} := \1_{\text{$(i,j)$ is an edge in $\sog(\Z,p)$}}.
\]
Then $(\alpha_{i,j})_{i<j}$ is a collection of i.i.d.\ Bernoulli random variables.
The case $p=0$ is trivial and shall not be considered.
The case $p=1$ corresponds to the full ordered graph $\sog(\Z,1)$ 
with  edges all the pairs $(i,j)$ with $i<j$ (there is 
nothing random in this graph).

In addition, we consider a pair $(u,v)$ of  independent random variables
that serve as edge and vertex weights, respectively.
In other words,
consider an array $(u_{i,j})_{i<j}$ of i.i.d.\ copies of $u$
and a sequence $(v_i)_i$ of i.i.d.\ copies of $v$.
We assume that the three sets, $(\alpha_{i,j})$, $(u_{i,j})$, $(v_i)$
are independent and let $\sog(\Z,p,u,v)$ denote the $\sog(\Z,p)$ with
weights $u_{i,j}$ added on each edge $(i,j)$ and $v_i$ on each vertex $i$.
The formal relation between the two is $\sog(\Z,p) = \sog(\Z,p,1,0)$.
We shall also consider the  auxiliary graph $\sog (\Z,p,u,0)$
with zero weights on the vertices.

A path $\pi$ in $\sog(\Z,p,u,v)$ (or, equivalently, in $\sog(\Z,p)$)
is a finite sequence of vertices $i_0 < i_1< \cdots < i_\ell$
such that $(i_{r-1},i_r)$ are edges, $r=1,\ldots, \ell$.
The path $(i_0, i_1,\ldots,i_\ell)$ is a path from $i$ to $j$ if $i_0=i$ and $i_\ell=j$.
Let $\Pi_{i,j}(p)$ be the set of paths from $i$ to $j$.
This set is random and may very well be empty.
If $\Pi_{i,j}(p)\not=\varnothing$ we say that $i$ and $j$ are connected
(and by this we always mean that the connection is via a path
from $i$ to $j$).

We define the weight of a path $\pi$ in $\Pi_{i,j}(p)$ by
\begin{equation}
\label{pathweight}
w(\pi)=\sum_{r=1}^\ell ( v_{i_{r-1}} + u_{i_{r-1},i_r}),
\quad \pi=(i_0,i_1,\dots,i_\ell) \in \Pi_{i,j}(p).
\end{equation}
In Section \label{secag} we will assume that $u$ can 
take negative values. Therefore, $w(\pi)$ can be negative.
We then
consider the maximization problem
\begin{equation}
\label{max1}
w_{i,j}:=\sup\{w(\pi):\, \pi \in \Pi_{i,j}(p)\}
\end{equation}
and set
\begin{equation}
\label{max1capital}
W_{i,j} := w_{i,j}^+.
\end{equation}

For the special case of $\sog(\Z,p,u,0)$ we let $\widehat w_{i,j}$, $\widehat W_{i,j}$
denote the quantities corresponding to \eqref{max1}, \eqref{max1capital},
respectively.

For the even more special case of $\sog(\Z,p,1,0)$ we let 
$w^0_{i,j}$, $W^0_{i,j}$
denote the quantities corresponding to \eqref{max1}, \eqref{max1capital},
respectively.

We are interested in asymptotic properties of $W_{i,j}$ as $|j-i| \to \infty$,
that is, a LLN and a CLT.
Despite the fact that there are $O(2^n)$ paths in $\Pi_{i,j}(1)$ when
$|j-i|=n$, the random weights are so highly correlated that we have
a linear asymptotic growth rate as $n \to \infty$, provided that
$\max (0,u)$  has a second moment and $\min (u,0)$ and $v$ a first.

A quick explanation  of this fact  
 is via an extended version
of the subadditive ergodic theorem.
Let $\WW_{i,j}$ be a related quantity, obtained by replacing
the maximization over all paths between two vertices in the segment $[i,j]$.
In other words,
\begin{equation}
\label{max2}
\WW_{i,j} = \max_{x,y \in [i,j]} W_{x,y}.
\end{equation}
It turns out that the value of $\WW_{i,j}$ is $W_{i,j}$ plus something
of  order  $o(n)$ as $n=|j-i| \to \infty$. This is partly due to
the fact (proved below, but also found in \cite{DFK12} and \cite{FMS})
that $i$ and $j$ are eventually connected with probability $1$.
It is clear that $(\WW_{i,j})_{i<j}$ is stationary, that is,
\[
(\WW_{i,j})_{i<j} \eqdist (\WW_{i+1,j+1})_{i<j}.
\]
In addition,
\[
\WW_{i,k}\le \WW_{i,j}+ \WW_{j,k} + \max_{i\le x\le j \le y\le k}
(v_x+u_{x,y})^+.
\]
An estimate for the first moment of the latter maximum shows
that it is finite iff the first moment of $v$ and the second moment of $u$
are finite.
An extended version of the subadditive ergodic theorem shows that
$\lim_{n \to \infty} W_{0,n}/n$ exists a.s., and, owing to ergodicity,
that it is a.s.\ equal to a constant (that can be seen to be positive).
Although this can provide a proof for the law of large numbers,
and, in fact, in a context much more general than the one considered here,
it gives no information about second-order properties.
So we bypass this avenue and consider instead discovering regenerative 
properties, as done in previous work.
The difference here is that edge-disjoint paths have correlated weights
(if they share common vertices) but we will see that this
does not complicate things much.

\section{Asymptotic growth}
We make the following assumptions concerning $(u,v)$:
\[
\P(v\ge 0)=1, \, \E v < \infty,\, \E u >0, \, \E\max (0,u)^2 < \infty.
\tag{A}\label{A}
\]
\label{secag}
This section is devoted to the proof of the following theorem.
\begin{theorem}
\label{thmLLN}
Consider the weighted random graph $\sog(\Z,p,u,v)$ with $0<p\le 1$,
and assume that conditions \eqref{A} 
hold.
Let $W_{i,j}$, $\WW_{i,j}$ be the values of the two optimization
problems \eqref{max1} and \eqref{max2}, respectively.
Then there is a constant $C>0$ such that
\[
\lim \frac{W_{i,j}}{j-i} = \lim \frac{\WW_{i,j}}{j-i} = C \text{ a.s.},
\]
as $j \to \infty$ or as $i \to -\infty$.
\end{theorem}
The method followed is that of exhibiting a regenerative structure
of a doubly-indexed process.
First, to fix ideas and notation, we define what we mean by
this term.

\begin{definition}
\label{defreg}
Let $\chi=(\chi_{i,j})_{i,j \in \Z, i<j}$ be an array of random elements 
defined on a common probability space $(\Omega, \FF,P)$,
and let $(A_i)_{i \in \Z}$ be a sequence of events.
Consider the random integers 
\[
\NN:=\{i:\, \1_{A_i}=1\}
\]
(the points $i$ such that $A_i$ occurs)
and enumerate them in some $\omega$-independent way. (For example,
let $\iota_1$ be the first $i > 0$ such that $\1_{A_i}=1$
and $\iota_0$ be the greatest $i \le 0$ such that $\1_{A_i}=1$
and enumerate the remaining points following their natural order.)
We then say that $\chi$ regenerates over $(A_i)$ (or over $\NN$) if
$(\chi_{i,j}:\, \iota_{r-1} \le i < j \le \iota_r)$, $r \in \Z$, are
independent (with the proper modification if the set  $\NN$ is finite).
\end{definition}
Once we have this definition in mind, the constructions below
will be clear. In fact we shall consider two sequences of events, the
skeleton points (denoted by $\mathcal S$) and the $c$-renewal points 
(denoted by $\mathcal R_c$). A third sequence will be considered in
the next section.
In Definition \ref{defreg} notice that if, in addition,
$\NN$ and $\chi$ are jointly stationary
then $\NN$ itself forms a stationary renewal process.
 
Define first the {\bf skeleton points}.\footnote{The terminology is from \cite{DFK12}. In \cite{FMS} the same points are called 
``strongly connected points''.} These depend only on connectivity
and not on weights. We say that $i$ is a skeleton point if it
connects to every point to the left and to the right:
\[
\SS=\{i \in\Z:\, \text{ there is a path between $i$ and any $j>i$
and between any $k<i$ and $i$}\}.
\]
As shown in \cite{DFK12},
\begin{lemma}
\begin{equation}
\label{skelrate}
\gamma:=
\P(i \text{ is a skeleton point})
= \prod_{k=1}^\infty (1-(1-p)^k)^2>0.
\end{equation}
\end{lemma}
\begin{proof}
For each $j \in \Z$, let $g_j$ be the distance from $j$ of the first $i < j$
such that $\alpha_{i,j}=1$.
We refer to $g_j$ as the \underline{first-left connection variable}.
Then $g_j$ is a geometric random variable with parameter $p$,
\[
\P(g_j > k) = (1-p)^k,
\]
and the $g_j$ are independent when $j$ runs over $\Z$.
We now notice the logical equivalence
\begin{equation}
\label{cruc}
0 \text{ connects to every $i$ in $\{1,\ldots,n\}$}
\iff g_1\le 1,\, g_2\le 2,\ldots,\, g_n \le n.
\end{equation}
Hence
\[
\P(0 \text{ connects to every $i > 0$})
= \prod_{k=1}^\infty \P(g_k \le k) 
= \prod_{k=1}^\infty (1-(1-p)^k).
\]
But for $0$ to be a skeleton point we need that it connects to
every point to its right and to its left. Hence $\gamma$ is the
square of the last quantity.
Finally, recall that for $a_k \in (0,1)$,
$\prod_{k=1}^\infty (1-a_k) > 0$ iff
$\sum_{k=1}^\infty a_k < \infty$, and this proves that $\gamma > 0$.
\end{proof}
In particular, we deduce that $\SS$ is an a.s.\ infinite random subset of $\Z$.
It is clear that it forms a stationary and ergodic point process.
What is not immediately clear is that 
\begin{lemma}
\label{lemS}
$\SS$ forms a stationary renewal process and 
$\chi=(\alpha_{i,j}, u_{i,j}, v_i)_{i<j}$ regenerates over $\SS$ in the sense of
Definition \ref{defreg}.
\end{lemma}
\begin{proof}[Sketch of proof]
Let $B_i$ be the event that there is a path from $i$ to any $j >i$
and from any $k < i$ to $i$. Recall that $\SS=\{i:\, \1_{B_i}=1\}$.
The {first thing} to prove is that, conditional on $B_i$, the 
future of $\chi$ after $i$ is independent of the past before $i$.
The  {second thing} to prove is that on $B_i$ 
the $\SS$-points to the right of $i$ are  completely determined by the future
of $\chi$ after $i$. Similarly from the past.
The crucial observation in proving these assertions is that
the event that $i$ is connected to every point $j \in [i+1,n]$
is determined by first-left connection variables; see \eqref{cruc}
Recall that the first left-connection variable $g_i$ is the smallest $k$
such that $(i-k, i)$ is an edge. 
Then the event that $i$ is connected to every point $j \in [i+1,i+n]$
is the event
\[
g_{i+1} \le 1, g_{i+2} \le 2,\ldots, g_{i+n} \le n.
\]
If we let $B_i^+$ be the event that $i$ is connected to every point to
its right then 
\[
B_i^+ = \{g_{i+1} \le 1, g_{i+2} \le 2,\ldots\}
\]
Therefore, if $B_0$ occurs then the event that the first $\SS$-point to the
left of $0$ is located at $k<0$ is the event 
$\{g_{k+1}\le 1,\ldots,g_0 \le |k|\}$ which is completely determined
by the past before $0$.
\end{proof}
\begin{corollary}
\label{coro1}
$(W_{i,j})_{i<j}$ regenerates over $\SS$.
\end{corollary}

To gain some intuition about the general case, 
we look at the special case of the graph $\sog(\Z,p)=\sog(\Z,p,1,0)$
and show that the asymptotic growth of the maximal path length follows
from  Lemma \ref{lemS} and Corollary \ref{coro1}.
\begin{proposition}[Special case of Theorem \ref{thmLLN}]
\label{speLLN}
Let $W_{i,j}^0$ be the value of the maximization problem \eqref{max1}
for $\sog(\Z,p,1,0)$, i.e., when $u=1$, $v=0$, a.s.
Then, as $j \to \infty$ or $i \to -\infty$,
\[
W^0_{i,j}/(j-i) \to C^0(p) \text{ a.s.},
\]
for some $C^0(p)>0$.
\end{proposition}
\begin{proof}[Sketch of proof]
If $\sigma \in \SS$ and $i \le \sigma \le j$, then, necessarily,
\begin{equation}
\label{Ssplit}
W^0_{i,j} = W^0_{i,\sigma} + W^0_{\sigma_,j}.
\end{equation}
Therefore, by Corollary \ref{coro1}, $W^0_{i,j}$ is the sum
of a number $M+1$ of random variables where $M$
is the number of skeleton points between $i$ and $j$. Since
$\SS$ has positive density $\gamma$, we have that
$M/(j-i) \to \gamma$ as $|j-i|\to \infty$.
We then obtain that $C^0(p)=\gamma \E W^0_{\sigma_1, \sigma_2}$,
where $\sigma_1, \sigma_2$ are two successive $\SS$ points to the right of $0$.
This constant is positive since $W^0_{\sigma_1,\sigma_2} \ge 1$.
\end{proof}
There is no closed form formula for $C^0(p)$. However, in \cite{FK03}, we 
obtained computable bounds for it by completely different methods.
More exact formulas have recently been obtained by Mallein and Ramassamy \cite{MR,MR2}.

To complete a revision properties of the graph $\sog(\Z,p)$, we
formulate the following result:

\begin{lemma}\label{allmoments}
Let $\sigma_1<\sigma_2< \cdots$ be the positive points of $\SS$.
Then $\sigma_{k+1}-\sigma_k$, $k=1,2,\ldots$, are i.i.d.\ and,
 for some $\theta>0$, $\E \exp \theta(\sigma_2-\sigma_1)<\infty$.
In particular, all moments of $\sigma_2-\sigma_1$ are finite.
\end{lemma}
We leave the proof for the reader. In what follows, we need only finiteness 
of the second moment, and this was proved in \cite{DFK12} in a more general setting
(the connectivity probability $p$ was allowed to depend on the distance between
the endpoints of an edge). 

In order to analyze the case of interest in this paper,
namely, the graph $\sog(\Z,p,u,v)$, we define a new set of points,
the {\bf $c$-renewal points}\footnote{The term ``renewal points'' was introduced in \cite{FMS}.}, where $c$ is a positive
constant. For that, we consider the  auxiliary directed graph
$\sog(\Z,p,u,0)$ and let $\widehat w_{i,j}$, $\widehat W_{i,j}$ denote the quantities 
in \eqref{max1}, \eqref{max1capital}, respectively, 
when all the $v_{i}$ are set equal to zero in \eqref{pathweight}.
Then the $c$-renewal points are defined as the points $i\in \Z$ at which that the events
\begin{align*}
&A_i^+:= \{\widehat{W}_{i,i+n} > cn \text{ for  all } n \ge 1\}
\\
&A_i^- := \{\widehat{W}_{i-n,i} > cn \text{ for all } n \ge 1\}
\\
&A_i^{-+} := \{ \alpha_{i-m,i+n}u_{i-m,i+n} < c (m+n)
 \text{ for all } m, n \ge 1\}
\end{align*}
occur simultaneously:
\[
\RR_c:=  \{i\in \Z:\, \1_{A_i^+} \1_{A_i^-} \1_{A_i^{-+}}=1\}.
\]
The three events
$A_i^+, A_i^-, A_i^{-+}$ are independent.
Indeed, they are functions of independent sets of random variables.

Points in $\RR_c$ achieve several things at the same time:
\\
\underline{First}, any point $i$ for which $A_i^+ \cap A_i^-$ holds is also
an $\SS$-point. So 
\[
\RR_c \subset \SS \text{ a.s.}
\]
\underline{Second}, any point $i$ for which $A_i^+ \cap A_i^- \cap A_i^{+-}$
holds ``splits'' the weighted graph $\sog(\Z,p,u,0)$ in the same way
that the $\SS$ points split the graph $\sog(\Z,p)$ [see \eqref{Ssplit}], that is, 
\begin{equation*}
\label{split-tilde}
i \in \RR_c,\, x<i<y \Rightarrow \widehat{w}_{x,y} = 
\widehat{w}_{x,i}+\widehat{w}_{i,y}.
\end{equation*}
Indeed, if $\pi$ is a path from $x$ to $y$ with weight $\widehat{w}_{x,y}$ such that
$i \not \in \pi$ then $\pi$ contains an edge $(a,b)$ with 
$x \le a < i < b \le y$. Since $A_i^{-+}$ holds, we have
$\alpha_{a,b}u_{a,b} < c(b-a) = c(i-a) + c(b-i)$ and since $A_i^-$ and $A_i^+$ 
hold we have $c(i-a) \le \widehat{w}_{a,i}$ and $c(b-i) \le \widehat{w}_{i,b}$.
Hence the weight of the edge $(a,b)$ can be strictly increased and this means
that $\widehat{w}_{x,y}$ can be strictly increased, contradicting
its optimality.\\
Since the $v$'s are non-negative, similar arguments work for the $w_{i,j}$,
and we have
\begin{equation}
\label{split}
i \in \RR_c,\, x<i<y \Rightarrow {w}_{x,y} = 
{w}_{x,i}+{w}_{i,y}.
\end{equation}
\\
\underline{Third}, for $c$ small enough, $\RR_c$ has positive density.
The reason for this is the law of large numbers related to the regenerative
structure over $\SS$ (that already has positive density). 
This is proved in Lemma \ref{lem3} below.
\\
\underline{Fourth}, the graph $\sog(\Z,p,u,v)$ regenerates over $\RR_c$
in the sense of Definition \ref{defreg}.
See  Lemma
\ref{lem4} below.

We let
\[
\lambda = \lambda(c) := \P(A_0^+ \cap A_0^- \cap A_0^{-+}).
\]
This quantity is the density of $\RR_c$. Our goal is to show that it is positive
for all small positive $c$.

\begin{lemma}
\label{lem3}
Assume that conditions \eqref{A} 
hold.
For all sufficiently small positive constants $c$ the random set 
$\RR_c$ has positive density.
\end{lemma}

\begin{proof}
Since the three events in the definition of $\RR_c$ are independent and since their intersections form a stationary ergodic sequence, the
 density of $\RR_c$ is the product of probabilities of these events.
So it is enough to show that each of these probabilities is strictly positive. To show that $\P(A_0^+)>0$ we use Corollary \ref{coro1}. We have
\[
\widehat{w}_{0,n} =  \widehat{w}_{0,\sigma_1}+ \widehat{w}_{\sigma_1,\sigma_2}+
\cdots + \widehat{w}_{\sigma_{M_n-1}, \sigma_{M_n}}+\widehat{w}_{\sigma_{M_n},n}, 
\]
where $M_n$ is the cardinality of $\SS \cap [1,n]$ and thus
$\sigma_1 < \cdots < \sigma_{M_n} \le n$.
Note that $M_n\to\infty$ a.s. Further,  due to
the regenerative structure, the
random pairs  $(\sigma_{k+1}-\sigma_k,
\widehat{w}_{\sigma_k, \sigma_{k+1}})$, $k=1,2,\ldots$ are  i.i.d.\ with
\[
\E(\sigma_{k+1}-\sigma_k,\,
\widehat{w}_{\sigma_k, \sigma_{k+1}}) = \left({\gamma}^{-1},\, \E\widehat{w}_{\sigma_1,\sigma_2}\right).
\]
By the SLLN and the integrated renewal theorem,
\[
\lim_{n \to \infty} \frac1n \sum_{k=1}^{M_n} 
\widehat{w}_{\sigma_{k-1}, \sigma_{k} }
= \lim_{n \to \infty} \frac{M_n}{n}\,
\frac{1}{M_n} \sum_{k=2}^{M_n} \widehat{w}_{\sigma_{k-1},\sigma_k}
=
\gamma \E \widehat{w}_{\sigma_1,\sigma_2} \text{ a.s.}
\]
Clearly, since $\widehat{w}_{0,\sigma_1}$ is a proper random variable, 
\[
\lim \widehat{w}_{0,\sigma_1}/n \to 0 \quad \mbox{a.s.}
\]
We next observe that
\[
\E \widehat w_{\sigma_i, \sigma_{i+1}}
\ge (\E W^0_{\sigma_i, \sigma_{i+1}}) \, \E u > 0.
\]
The reason for the first inequality is that 
$W^0_{\sigma_i, \sigma_{i+1}}$ is the length of the longest path from
$\sigma_i$ to $\sigma_{i+1}$ and the edge weights $(u_{i,j})_{i<j}$
are independent of $(\alpha_{i,j})_{i<j}$.
Consider the nonnegative random variables
\[
Y_i := \sum_{\sigma_i \le j_1 < j_2 \le \sigma_{i+1}} u_{j_1,j_2}^-,
\]
where $x^- = (-x)^+ = -\min(0,x)$.
Note that
\[
\E Y_i \le  \tfrac{1}{2} \, \E (\sigma_{2}-\sigma_i)^2\, \E u^- < \infty,
\]
where the finiteness comes from Lemma \ref{allmoments} and assumption \eqref{A},
and that
\[
\widehat w_{\sigma_i,\sigma_{i+1}} \ge -Y_i.
\]
We thus have
\[
\liminf_{n \to \infty} \frac{\widehat w_{\sigma_{M_n},n}}{n}
\ge \liminf_{n \to \infty} \frac{M_n}{n} \, \frac{-Y_{M_n}}{n} = 0 \cdot \gamma = 0,
\quad \text{a.s.}
\]
Let now $c$ be such that 
\[
0 < c < (\E W^0_{\sigma_1, \sigma_{2+1}}) \, \E u.
\]
It is then a simple consequence of the ergodic theorem that
there exists $n_0$ such that
\[ 
\P(\widehat{W}_{0,n} > n c \text{ for all } n\ge n_0) > 0.
\]
We conclude that
\begin{multline*}
\P(A_0^+) \ge \P\big(
\big\{\text{$(j_1,j_2)$ is an edge}, u_{j_1,j_2}>c, \ \text{for all} \ 0\le j_1<j_2\le n_0\big\} 
\\
\cap 
\big\{\widehat{W}_{0,n} > n c \text{ for all } n\ge n_0\big\}\big) >0
\end{multline*}
because the two events in the intersection are positively correlated.

By symmetry, $\P(A_0^-)>0$ as well. 

It remains to show that
$\P(A_0^{-+})>0$.
To this end, we let
\[
U_k \eqdist \alpha_ku_k
\]
where the $\alpha_k$ are i.i.d.\ Bernoulli($p$) random variables, the $u_k$ are
i.i.d.\ copies of $u$
and $(\alpha_k)$ and $(u_k)$ 
are independent.
Taking into account the independence between the variables involved
in the definition of $A_0^{-+}$, we have
\[
\P(A_0^{-+}) = \prod_{m=1}^{\infty} \prod_{n=1}^{\infty}
\P(U_{m+n}<c(m+n)) =
\prod_{r=2}^{\infty} \P(U_1<ck)^{k-1}.
\]
Note that, for any $x>0$, we have 
\[
\P(U_1<x) = 1-p+p\P(u_1<x) = 1-p\P(u_1\ge x).
\]
Therefore,
\[ 
\P(A_0^{-+}) >0
\]
since $E \max (0,u)^2 <\infty$ and, then,  
\[ 
\prod_{k=2}^{\infty} \left( 1-p\P(u_1\ge kc)\right)^{k-1} \ge
K \exp (-\sum_k k\P(u_1\ge kc)) >0,
\]
for a certain constant $K>0$. 

Therefore the density $\lambda$ of $\RR_c$, for all sufficiently small positive $c$,
satisfies
\begin{equation}
\label{lala}
\lambda = \P(A_0^+ \cap A_0^- \cap A_0^{-+}) 
= \P(A_0^+) \P(A_0^-) \P(A_0^{-+}) > 0.
\end{equation}
\end{proof}

\begin{lemma}
\label{lem4}
$\chi=(\alpha_{i,j},\, u_{i,j},\, v_i)_{i<j}$ regenerates over $\RR_c$.
\end{lemma}
\begin{proof}
It suffices to show that, conditional on $A_i$, the future of $\chi$ after $i$ is
independent of the past, and that, on $A_i$, the future (respectively, past) 
$\RR_c$-points depend only on the future  (respectively, past)  of $\chi$.
Without loss of generality, let $i=0$. 
Let $\FF^+$ be the $\sigma$-algebra generated by $\chi_{i,j}$, $0 \le i < j$.
Similarly, we define $\FF^-$ as 
the $\sigma$-algebra generated by $\chi_{i,j}$, $i<j \le 0$. 
The two $\sigma$-algebras are independent. Now let $\GG^+ \subset
\FF^+$, $\GG^- \subset\FF^-$ be two sub-$\sigma$-algebras
and let $\GG$ be a $\sigma$-algebra independent of $\FF^+$
and $\FF^-$. It is easy to see that $\FF^+$ and $\FF^-$
are independent conditionally on $\GG^- \vee \GG^+ \vee \GG$.
Apply this observation to $\GG^+ =\sigma(A_0^+)$,
$\GG^- =\sigma(A_0^-)$ and $\GG=\sigma(A_0^{-+})$.
We next establish that, on $A_0$, the $\RR_c$ points to the right of $0$
depend only on variables $\chi_{i,j}$, $0 \le i < j$ (and, similarly, for the past).
This is equivalent to showing that, on $A_0$, for each $k>0$, the event $A_k$
depends only on variables $\chi_{i,j}$, $0 \le i < j$.
Recall that
\[
A_0 = \bigcap_{m\ge 1, n \ge 1}
\{\widehat{W}_{-m,0} \ge cm,\, 
\widehat{W}_{0,n} \ge cn,\, U_{-m,n} < c(m+n)\},
\]
where 
\begin{equation}
\label{UUU}
U_{i,j} := \alpha_{i,j}u_{i,j},
\end{equation}
and that
\[
A_k = \bigcap_{m\ge 1, n \ge 1}
\{\widehat{W}_{k-m,k} \ge cm,\, 
\widehat{W}_{k,k+n} \ge cn,\, U_{k-m,k+n} < c(m+n)\}.
\]
Consider the truncated event
\[
\widetilde A_k = \bigcap_{1 \le m\le k, n \ge 1}
\{\widehat W_{k-m,k} \ge cm,\, \widehat W_{k,k+n} \ge cn,\, U_{k-m,k+n} < c(m+n)\}
\]
for which we have $\widetilde A_k \in \FF_+$.
Our claim will follow from the identity
\[
A_0 \cap A_k = A_0 \cap \widetilde A_k.
\]
Since $A_k \subset \widetilde A_k$ we only have to show that
$A_0 \cap \widetilde A_k \subset A_0 \cap A_k$.
Suppose that $A_0$ and $ \widetilde A_k$ occur. 
We need to show that $\widehat W_{k-m,k} \ge m$ for all $m > k$
and that $U_{k-m,k+n} < c(m+n)$ for all $m>k$ and $n \ge 1$.
Let $m >k$. Then $k-m < 0 < k$.
Since $A_0$ holds, we have $\widehat W_{k-m,k} \ge \widehat W_{k-m,0} + \widehat W_{0,k}$.
But $\widehat W_{k-m,0} \ge c(m-k)$ (because $A_0$ holds)
and $\widehat W_{0,k} \ge ck$ (because $\widetilde A_k$ holds).
Hence $\widehat W_{k-m,k} \ge c(m-k)+ck=cm$, as required.
Let again $m>k$ and let $n \ge 1$. Then $k-m < 0 < k+n$
and $U_{k-m, k+n} < c(m+n)$ because $A_0$ holds.
\end{proof}

\begin{corollary}
\label{coro2}
Under the assumptions of Lemma \ref{lem3},
$\RR_c$ is a stationary renewal process with
rate $\lambda>0$ as in \eqref{lala}
and both $(\widehat{W}_{i,j})_{i<j}$ and $(W_{i,j})_{i<j}$ regenerate over $\RR_c$.
\end{corollary}

\begin{proof}[\bf Proof of Theorem \ref{thmLLN}]
Let $c$ be chosen as in Lemma \ref{lem3}.
By Lemma \ref{lem3}, $\RR_c$ has positive density. 
By Lemma \ref{lem4}, we have regeneration over $\RR_c$. 
By \eqref{split}, the maximal path from some $i$ to some $j>i$
can be written as a sum of independent finite-mean  random variables.
Let $\tau_1 < \tau_2 < \cdots$ be the points of
$\mathcal R_c \cap (0,\infty)$. 
The LLN $\lim_{n \to \infty} W_{0,n}/n = C$ a.s., 
with $C = \lambda^{-1} \E[W_{\tau_1, \tau_2}]$,
then follows from a standard renewal theory argument.
\end{proof}

\section{Central limit theorem}
 For simplicity we now assume that both $u$ and $v$ are nonnegative;
but it is essential that they satisfy more stringent moment conditions.
Our assumptions for this section are then
\[
\P(u\ge 0, v\ge 0)=1, \, \E v ^2< \infty,\, \E u >0, \, \E u^3 < \infty.
\tag{B}\label{B}
\]

\begin{theorem}
\label{thm2}
Consider the weighted random ordered graph $\sog(\Z,p,u,v)$ with $0<p\le 1$, assume that conditions \eqref{B} hold.
Let $W_{i,j}$ be the values of optimization
problem \eqref{max1}. 
Let $C$ be the constant appearing in Theorem \ref{thmLLN}.
Then there is a constant $b>0$
such that
\[
\frac{W_{0,[nt]} - C nt}{b \sqrt{n}},\, t \ge 0,
\]
converges weakly, as $n \to \infty$, to a standard Brownian motion.
\end{theorem}
Note that by weak convergence we mean convergence of the
law of the process above, considered as a random element of the space
$D[0,\infty)$ equipped with the topology of uniform convergence
on compact sets.

%
%
%
%

\begin{proof}[Sketch of proof of Theorem \ref{thm2}]
The essential part is in proving the ordinary CLT, that is,
$(W_{0,n}-Cnt)/ \sqrt{n}$ converges in distribution to
a  normal random variable with positive variance $b^2$.
The passage from the CLT to the  functional version stated in the theorem 
is standard.
The difficulty in proving the ordinary CLT is in establishing that the 
variance between successive positive $\RR_c$--points is finite.
Once this is established, standard renewal theory shoes that
finiteness of variance between successive positive $\RR_c$--points
is equivalent to finiteness of  expectation of the first positive $\RR_c$--point.
We shall show this by constructing an upper bound.
We follow ideas in \cite{FMS}.

Let $\UU_c$ be the random set containing $i$ such that
$A_i^-$ occurs. Just as is Lemma \ref{lem4}, we have that our random
structure $\chi$ regenerates over $\UU_c$. We have $\RR_c \subset
\UU_c$ and, for $c$ small enough (Lemma \ref{lem3})
$\UU_c$ has positive density.
In particular, $\UU_c$ is a stationary renewal process.

Moreover, the variance between two successive points points of $\UU_c$
is finite.
To show this, it suffices to show (since $\UU_c$ is a stationary renewal process)
that the first positive point of $\UU_c$ has finite expectation provided
that 
\[
c < \gamma\, \E \min_{\sigma_1 \le i<j \le \sigma_2} (u_{i,j}).
\]
In fact, more is true: under this assumption, 
the first positive point of $\UU_c$ has an exponential moment.
We skip the proof as it is analogous to the proof of
Proposition 3.12 in \cite{FMS}.

Define events analogous to $A_i^+$, etc.
\[
A_{i,d}^+ := \{\widehat{W}_{i,i+n} > cn,\, 1 \le n \le d\}
\]
\[
A_{i,d}^- := \{\widehat{W}_{i-n,i} > cn,\, 1 \le n \le d\}
\]
\[
A_{i,d}^{-+} := \{U_{i-m,i+n} < c(m+n),\, 1 \le m \le d,\, n \ge 1\},
\]
where the $U_{i,j}$ are as in \eqref{UUU}.
Notice that, as $d \to \infty$,  the event $A_{i,d}^+$ decreases to $A_i^+$,
and similarly for the other two events.

Consider now the random variable
\[
\mu:=\inf\{d>0:\, \1_{A_{0,d}^+ \cap A_{0,d}^{-+}}=0\}.
\]
Note that
\[
\P(\mu=\infty) = \P(A_{0}^+ \cap A_{0}^{-+}) >0,
\]
for $c$ sufficiently small.
By an estimate similar to the one performed in the proof
of Lemma \ref{lem3}, we see that if conditions \eqref{B} are satisfied
then
\[
\E(\mu| \mu < \infty) < \infty.
\]
For $n>0$ (perhaps random), let $\theta^n \mu$ be obtained in the same
manner as $\mu$ after shifting the origin at $n$, namely, let
\[
\theta^n \mu:= \inf\{d>0:\, \1_{A_{n,d}^+ \cap A_{n,d}^{-+}}=0\}.
\]
We now consider the following algorithm:
\begin{itemize}
\item[1)] Initialize by letting $\psi_0$ be the first positive point of $\UU_c$.
\item[2)] Suppose that $\psi_0,\ldots, \psi_k$ have been defined.
If $\theta^{\psi_k}\mu < \infty$ let $\psi_{k+1}$ be the smallest
point of $\UU_c$ to the right of $\psi_k+\theta^{\psi_k}\mu$.
Otherwise, if $\theta^{\psi_k}\mu=\infty$, let $\psi_{k+1}=+\infty$,
set $\Psi:= \psi_k$ and stop.
\end{itemize}

Clearly, $\Psi$ is an upper bound to the first positive $\RR_c$ point. 
Taking into account the regenerative structure, we easily see that $\E\Psi<\infty$
if the first positive point of $\RR_c$ has finite expectation and if 
$\E(\mu|\mu<\infty)<\infty$.
Since the first holds for sufficiently small $c$ while the second holds because
$\E u^3 < \infty$, $\E v^2$, by assumption,
we have established that the first positive point of $\RR_c$ has finite expectation
and thus, by renewal theory, that the distance between successive
positive points of $\RR_c$ has finite variance.

Using the established fact that $\RR_c$ has finite variance it can be shown,
just as in Proposition 3.14 of \cite{FMS}, that
\[
b^2 := \var(W_{\Gamma_1,\Gamma_2}-C(\Gamma_2-\Gamma_1)) < \infty,
\]
where $\Gamma_1, \Gamma_2$ are the first two positive $\RR_c$-points.
The CLT now follows, from the standard regenerative CLT.

\end{proof}


\section{Convergence to the continuum cascade model}
Consider the  random graph $\sog(\Z,p,u,v)$ when $p$ is small.
Notice that many events of interest, such as ``the longest path from $0$ to a vertex $i>0$'',
depend only on the component of the graph that contains vertex $0$ (that is,
on the subgraph consisting of all vertices reachable from $0$);
let $\sog^0(\Z,p,u,v)$ denote this component.
Computing the probability of such events is hard. 
However, a reasonable approximation, when $p \to 0$, can be obtained
by showing that the component $\sog(\Z,p,u,v)$, as a random element of a suitable
Polish space, converges weakly to a random object that is a weighted tree
and, therefore, recursive equations can be written for the probabilities/quantities
of interest. We shall denote the limiting random object by $\ccm^0(u,v)$.
In the sequel, we first give the appropriate definitions of the random objects involved,
we describe the metric space in which they take values, we point out that this space
is a complete separable metric space under suitable metrics, and we finally sketch 
the proof of the convergence in distribution.

\paratitle{Scaling and the proposed limit}
We wish to consider a certain limit of the weighted random graph 
with weights on the edges and vertices when $p \to 0$. 
Let us consider a sequence of graphs indexed by positive
integers $n$ such that $p_n \to 0$ in a way that $n p_n$ converges
to a positive constant, say, $1$. 
Take as set of vertices the set $\frac{1}{n} \Z$ and declare
$(i/n, j/n)$, $i<j$, as an edge with probability $p_n$.
Let the edge and vertex weights be equal in distribution to
$u_n$ and $v_n$ respectively. We thus consider a weighted random graph 
that, following earlier notations, we denote as
$G_n = \sog(\frac1n\Z, p_n, u_n, v_n)$.

The continuum cascade model (CCM) \cite{IK12} is defined as follows:
Let $N$ denote a stationary Poisson(1) process on the real line.
Let, for each $t \in \R$, $N_t$ be equal in distribution to $N$
and assume that the collection $(N_t, t \in \R)$ is independent.
The CCM is a random graph with vertices $\R$.
For $s, t \in \R$, $s< t$, we declare that $(s,t)$ is an edge if $t$ is 
a point of the Poisson process $N_s$. Let $u, v$ be positive
random variables. We attach i.i.d.\ weights to the edges of the CCM 
that are all equal to $u$ in distribution. Similarly, we 
attach i.i.d.\ weights to the vertices of the CCM 
that are all equal to $v$ in distribution. Denote the weighted CCM
by CCM$(u,v)$.
(We shall leave to the reader to show that
all functions of the $\ccm(u,v)$ that we use in the sequel of the paper are measurable
functions of it.)
A path in the CCM is a finite or infinite
sequence $t_0 < t_1 < \cdots$ of vertices
such that $t_k$ is a point of $N_{t_{k-1}}$, $k=1,2,\ldots$
The weight of a finite path $(t_0, \ldots, t_m)$ is
the sum of the weights of its first $m-1$ vertices and its $m-1$ edges.

Let CCM$^0(u,v)$ be the restriction of the CCM$(u,v)$ on the vertices that
are reachable from $0$ via paths. Clearly, the CCM$_0(u,v)$ 
has countably many vertices and countably many edges
(but it has finitely many vertices on every bounded interval).
The set of vertices is a countable random subset of $\R_+$.
In fact,  CCM$^0(u,v)$ is a tree and it shall be considered as rooted
at $0$.

Going back to $\sog(\frac1n\Z, p_n, u_n, v_n)$ and forgetting the directions
of the edges, we see that it is a strongly connected graph and certainly
not a tree: the very existence of skeleton points  creates
cycles. However, the rate of skeleton points, see equation \eqref{skelrate},
can be shown
to satisfy
\[
\log \gamma_n \asymp - c_1 n
\]
for some positive constant $c_1$.
The intuition gained from this is that cycles vanish in the limit
and that the graph becomes a forest (a collection of trees).

Let $G^0_n$ be the restriction of $G_n$ on those vertices to the right
of the origin that are reachable from the origin. The only discrepancy
between $G^0_n$ and $G_n$ is between $0$ and the first skeleton
point to the right of $0$. (But the first skeleton point tends to $\infty$
as $n \to \infty$.) Our goal is to show that
\[
G^0_n \convd \text{CCM}^0(u,v) \text{ as } n \to \infty,
\]
where $\convd$ denotes convergence in distribution in a certain sense defined below. Once we have this convergence rigorously defined and
proved, we also have convergence in distribution of interesting functionals
of the graph to the corresponding functional of the limiting object.

\paratitle{Metrics}
A \underline{weighted geometric graph (wgg)} $G=(V, E, \ell, u, v)$ is a graph $(V,E)$
in the ordinary sense and three weight functions:
$\ell : E \to \R_+, u : E \to \R_+$ are weights on edges and $v: V \to \R_+$
are weights on vertices. The difference between $\ell$ and $u$ is that
$\ell$ is used to define distances between vertices:
\[
\ell(x,y) = \inf_{\pi \in \Pi_{x,y}} \ell(\pi),
\]
where $\Pi_{x,y}$ is the set of paths from $x$ to $y$ and where the 
\underline{length}
$\ell(\pi)$ of a path $\pi=(x_0, x_1, \ldots,x_m)$
is taken to be $\ell(x_0,x_1)+\cdots
+\ell(x_{k-1},x_k)$. This defines an extension of the original $\ell$ to
all pairs of vertices. The extension is a metric on $V$. Denote by
\[
B_r(x) := \{y \in V:\, \ell(x,y) \le r\}
\]
the ball of radius $r$ centered at $x$.
We call $G$ {\em locally finite} if $B_r(x)$ is a finite set for all $x$ and $r$.
The \underline{weight} of a path $\pi=(x_0, x_1, \ldots,x_m)$ is defined as
$\sum_{i=1}^m [v(x_{i-1})+u(x_{i-1},x_i)]$.
If $G, G'$ are \underline{finite} wgg's we define
\[
d(G,G') = \min_\phi \max_{e\in E,x\in V}\big\{
|\ell(e)-\ell'(\phi(e))| + |u(e)-u'(\phi(e))| + |v(x)-v'(\phi(x)| \big\},
\]
where $\phi$ runs over all bijections from $V$ to $V'$ that
preserve the edge structure: $e \in E \iff \phi(e) \in E'$.
If $G, G'$ are locally finite 
infinite wgg's then we define a distance between them
from the point of view of specific vertices, $0,0'$, say, that we call roots.
Let $G^{(r)}$ be the restriction of $G$ on the set of vertices $B_r(0)$ in
the obvious sense. Defining $d$ as in the last display but also
further restricting $\phi$ to be such that $\phi(0)=0'$, we let
\[
D(G,G') := \int_0^\infty \big(1 \wedge d(G^{(r)}, G'^{(r)})\big)
\, e^{-r} dr.
\]
It is easy to see that $D$ is a metric on the set $\GG_*$ of locally
finite wgg's.\footnote{more specifically, $\GG_*$ should be taken
to be  a set of  rooted locally finite
wgg's whose sets of vertices vary on the set of subsets of a universal set}
This metric is just an extension of the one proposed by Aldous and Steele
\cite{AS}. Adopting the proof of \cite[Prop.\ 2]{GAB}, one easily has
that $\GG_*$ is a Polish (complete separable metric) space.
Intuitively, $D(G,G')$ is small if the wgg's look similar (identical, up to
isomorphism, as logical graphs and with comparable lengths and weights)
on every finite-radius ball around the root.
We therefore have the full machinery of convergence of probability measures
\cite{BILL68} on Polish spaces available. In particular,
if $G_n, G$ are random elements of $\GG_*$, the convergence
$G_n \convd G$ is equivalent to $\E f(G_n) \to \E f(G)$ for any
bounded continuous function $f: \GG_* \to \R$.

Consider now the weighted random directed 
graph $G_n^0$ as defined earlier.
Think of it as a random wgg with root $0$ and edge length equal to
their physical distance, i.e., take the length of $(i/n,j/n)$,
$i<j$, to be $(j-i)/n$.
Similarly, consider $\ccm^0(u,v)$ as a random wgg with root $0$
and, if $(s,t)$ is an edge, let its length be $t-s$.

\begin{theorem}\label{convd}
For $n \in \N$, let $u_n, v_n$ be positive random variables,
$0<p_n<1$ and let 
$G_n^0$ be the weighted random directed 
graph $\sog(\frac1n\Z, p_n, u_n, v_n)$
restricted on the set of vertices reachable from $0$.
Assume that, as $n \to \infty$, $n p_n \to 1$, $u_n \convd u$,
$v_n \convd v$.
Let $\ccm(u,v)$ be the weighted continuum cascade model
on $[0,\infty)$ with i.i.d.\ edge weights distributed according to $u$
and i.i.d.\ vertex weights distributed according to $v$,
and let $\ccm^0(u,v)$ be its restriction on the set of vertices reachable from
$0$. Then
\[
G^0_n \convd \text{CCM}^0(u,v) \text{ as } n \to \infty.
\]
\end{theorem}

\paratitle{The weighted PWIT and the scaling limit theorem}
Notice that $G^0_n$ is a wgg with deterministic vertices, while $\ccm^0(u,v)$ is a wgg with random vertices. We get an easier 
handle of both objects if we construct them as (deterministic) functions
of the same logical tree. This is the Harris-Ulam tree recalled below.

Let $\N^* := \bigcup_{n=0}^\infty \N^n$ be the set of all finite
sequences of integers (integer words), where $\N^0:=\{\varnothing\}$
is the singleton containing the empty sequence, and equip it with 
concatenation: if $x,y \in \N^*$ then their concatenation $xy$ is
a finite sequence obtained by appending $y$ to $x$. (The empty sequence
is a neutral element.)
Form a graph on $\N^*$ by considering as an edge any $(x,y)$
such that $y$ is obtained by concatenating $x$ with a single integer,
$y=xk$, $k \in \N$, say. 
Let $E(\N^*)$ be the set of edges.
This is a countably infinite, locally finite,
infinitary tree, often known as the Harris-Ulam tree. We take the
empty sequence $\varnothing$ as its root.

We can make $\N^*$ a wgg by considering functions
$\ell, u: E(\N^*) \to \R_+$, $v: \N^* \to \R_+$ and use $\ell$
as a length function on its edges and $u, v$ as weight functions on
its edges and vertices, respectively. We denote the wgg thus obtained
by $(\N^*, \ell, u, v)$. The difference between $\ell$ and $u$
is that $\ell$ is used to define a metric on $\N^*$, whereas $u$ is
additional decoration. 

We also define a map
\[
\mathcal C :  (\N^*, \ell, u, v) \mapsto (V', E', \ell', u', v'),
\]
calling it  \underline{collapse map}, that takes 
the weighted geometric tree $(\N^*, \ell, u, v)$ onto some weighted 
geometric graph  $(V', E', \ell', u', v')$ with $V'$ a certain subset of $\R$. 
The definition of $\mathcal C$ is straightforward: The root of $\N^*$
is mapped to $0$.
We take $V'$ to be all $t \in \R_+$ such that there is $x \in \N^*$
with $\ell(x,\varnothing)=t$. We let $E'$ be all $(s,t)$, $0\le s < t$,
such that $s=\ell(x,\varnothing)$, $t=\ell(xk,\varnothing)$, for some
$x \in \N^*$ and $k \in \N$. If $(s,t) \in E'$ we let $\ell'(s,t)=t-s$.
Basically, $(V',E')$ is the ``shadow'' of $\N^*$ when vertices are placed
at the correct distance from the origin. Note that vertices in $V'$
may have multiple preimages in $\N^*$. 
We finally assign weights to $V'$ and $E'$.
For each $s \in V'$ let $x$ be the lexicographically least $x \in \N^*$
such that $\ell(x,\varnothing)=s$ and let $v'_s$ be equal to $v_x$.
If, in addition, $t$ is such that $(s,t) \in E'$ choose $x$ as above,
let $k$ be such that $\ell(x,xk)=t-s$ and give $(s,t)$
weight $u'_{s,t} = u_{x,xk}$.
\begin{lemma}
\label{lemcollcont}
The collapse map $\mathcal C$ is continuous in the 
metric $D$ of $\GG_*$.
\end{lemma}

The following lemma
is stronger than what we need here.
\begin{lemma}
\label{lemstrong}
Let $(\ell, u, v)$, $(\ell_n, u_n, v_n)$, $n \in \N$, be a random elements of 
$(\R_+^{E(\N^*)}, \R_+^{E(\N^*)}, \R_+^{\N^*})$
such that $(\ell_n, u_n, v_n) \convd (\ell, u, v)$, in the sense
of convergence of finite-dimensional distributions.
Then $(\N^*, \ell_n, u_n, v_n) \convd (\N^*, \ell, u, v)$ as
random wgg's.
\end{lemma}

We now construct a specific $T=(\N^*, \ell, u, v)$.
Let $\Phi_x$, $x \in \N^*$, be an i.i.d.\ collection of Poisson(1) processes
on $(0,\infty)$. That is, let $\tau_{xk}$, $x \in \N^*$, $k \in \N$,
be i.i.d.\ exponential(1) random variables. The $k$-th point of
$\Phi_x$ is a.s.\ equal to $\tau_{x1}+\cdots+\tau_{xk}$ and we let
\begin{equation}
\label{defell}
\ell(x,xk):=\tau_{x1}+\cdots+\tau_{xk}
\end{equation}
be the length of the edge $(x,xk)$. Without adding any extra weights on
edges and vertices, the random tree, $(\N^*,\ell)$, thus formed
is a random rooted geometric graph known as the {\em Poisson Weighted
Infinite Tree} (PWIT), introduced by Aldous and Steele \cite{AS}.
Add next i.i.d.\ weights to the edges and vertices of $\N^*$ distributed
according to $u$ and $v$, respectively. We let $T$ denote
 random  rooted wgg tree thus obtained and refer to it as the 
$\pwit(u,v)$. Observe now that
\begin{lemma}
\label{lemT}
With $T$ being the weighted $\pwit(u,v)$ we have 
$\ccm^0(u,v)=\mathcal C(T)$.
\end{lemma}

On the other hand,
let, for each $n$, $g^{(n)}$ be a geometric random variable
with parameter $p_n$, that is, $\P(g^{(n)}>k) = (1-p_n)^k$, $k\in \N$.
Take a collection $(g^{(n)}_{xk})_{x \in \N^*, k \in \N}$ of
i.i.d.\ copies of $g^{(n)}$ and define lengths on the Harris-Ulam tree by
\begin{equation}
\label{defelln}
\ell_n(x,xk) := \frac{g^{(n)}_{x1}+\cdots+g^{(n)}_{xk}}{n},
\quad x \in \N^*,\, k \in \N.
\end{equation}
Again, add  i.i.d.\ weights to the edges and vertices of $\N^*$ distributed
according to $u_n$ and $v_n$, respectively. Let $T_n=(\N^*,\ell_n,u_n,v_n)$ denote
the random  rooted wgg tree thus obtained. 
\begin{lemma}
\label{lemTn}
With $T_n$ being the tree just defined we have 
$G_n^0=\mathcal C(T_n)$.
\end{lemma}

\begin{proof}[Proof of Theorem \ref{convd}]
By Lemma \ref{lemstrong}, we have $T_n \convd T$ and this is
simply because, with $g^{(n)}$ geometric with
parameter $p_n$, and $\tau$ exponential with parameter 1,
we have  $g^{(n)}/n \convd \tau$. 
Then, from \eqref{defell} and \eqref{defelln},
we see that $\ell_n(x,xk) \convd \ell(x,xk)$ and, by independence, we
see that the conclusion of Lemma \ref{lemstrong} holds.
Since  (Lemma \ref{lemcollcont}) the map $\mathcal C$ is continuous,
we have that $\mathcal C(T_n) \convd \mathcal C(T)$.
This reads $G^0_n \convd \ccm^0(u,v)$ because of 
Lemmas \ref{lemT} and \ref{lemTn}.
\end{proof}

\begin{remark}
 We did not attempt to define the limit of the
full $G_n=\sog(\frac1n \Z, p_n, u_n, v_n)$ graph, but just of
its the graph $G_n^0$ 
obtained by those vertices that are reachable from $0$.
\end{remark}

\paratitle{Recursive distributional equation for 
maximal weight}
Having established a weak convergence result for $G^0_n$,
with $\ccm^0(u,v)$ as a limit,
we can now apply it to various interesting functionals,
so long as these functionals are continuous.

As an example, consider $\ccm^0(u,0)$.  (Set all vertex weights equal to zero
but let edge weights be i.i.d.\ all distributed as $u$.) Let
\[
\widetilde W_t := \text{ max weight of
all paths in $\ccm^0(u,0)$ starting from $0$ and having length at
most $t$.}
\]
We then have
\[
\widetilde W_t \eqdist \max_{1 \le i \le N_t}
\big\{ u_i +\widetilde W_{x-T_i} \big\},
\]
where $N$ is a Poisson process on $[0,\infty)$,
independent of $\widetilde W$, with points $0<T_1<T_2<\cdots$,
and where $N_t$ is the number of points on $[0,t]$; $T_{N_t} \le t$.
Since, conditional on $\{N_t=k\}$,
the set $\{T_1, \ldots,T_k\}$ is a set of $k$ i.i.d.\ uniform
random variables with values in $[0,t]$, we obtain
\begin{equation}
\label{FTW}
F(t,w) := \P(\widetilde W_t > w)
= \exp \int_0^x \E[F(y, w-u)]\, dy.
\end{equation}
This equation, in the special case where $u=1$ a.s., has been
derived in \cite{IK12}, along with some heuristics on
its behavior for large $t$.

\section{Some open problems}

\resetpara

\para
Motivated by the results in \cite{JOHANSSON00} and \cite{KT13}
we may ask a last-passage percolation question for the following model,
a 2-dimension generalization of the continuum cascade model.
Let $\Phi$ be a unit-rate Poisson process on the positive orthant
$\R_+^2$. Let $(\Phi_t, \, t\in \R_+^2)$ be i.i.d.\ copies of $\Phi$.
For $s=(s_1,s_2)$, $t=(t_1,t_2) \in \R_+^2$, write
$s<t$ if $s_1 < t_1$ and $s_2 < t_2$.
Now define a graph with vertices in $\R_+^2$ and edges $(s,t)$
provided that $s< t$ and that $t$ is a point of $\Phi_s$.
Let $L_{x,y}$ be the maximum length of all paths starting from $0$
and contained in the rectangle $[0,x] \times [0,y]$.
The question is the asymptotic growth of $L_{x,y}$ and its
fluctuations.

\para
Consider the graph $\sog(\Z,p, u, 0)$, for a reasonable edge
weight variable $u$. The case $u=1$ and $p>1/2$ has been
settled in \cite{MR}. However, if $u$ is not deterministic, and even
in the simplest possible case where $u$ takes 2 values only,
it appears that even bounds on the asymptotic growth of the maximum
length are not easy to get. For example, the method of \cite{FK03}
fail.

\para
A discrete model that simultaneously generalizes that of \cite{JOHANSSON00} 
and \cite{KT13} is as follows.
Consider last passage percolation on $\Z_+^2$ with vertex weights as in
\cite{JOHANSSON00} but allow the possibility of random edges
as well as in \cite{KT13}. Do the fluctuations of maximal path lengths
also have a Tracy-Widom limit in distribution?


\para
Equation \eqref{FTW} provides an approximation for the distribution 
heaviest path
in the discrete graph $\sog(\Z,p,u,0)$ in the small $p$ regime.
It is a very complicated integral fixed-point equation whose properties
are not understood. In the $u=0$ case (no weights at all), heuristics for its
solution are in \cite{IK12}. A similar fixed-point equation can be derived
for the general $(u,v)$ case.

\section*{\normalsize Acknowledgments}
The second author would like to acknowledge the hospitality of the
Statistics Department at the University of California, Berkeley
(and, in particular, David Aldous)
in July 2017, where part of this research was done.
We would also like to thank the referees for pointing out a blunder and
some unclear sentences in the original version.

\bibliography{weightedsog13_arxiv}{}
\bibliographystyle{plain}

\end{document}